\documentclass[numbers,webpdf]{ima-authoring-template}%
\journaltitle{...}
\usepackage{hjk}

\title{On network dynamical systems with a nilpotent singularity}

\author{Hildeberto Jard\'on-Kojakhmetov
\address{
\orgdiv{Dynamical Systems, Geometry \& Mathematical Physics},
\orgname{University Of Groningen},
\orgaddress{\street{Nijenborg 9}, \postcode{9747 AG}, \state{Groningen},\country{The Netherlands}}
}
} 

\author{Christian Kuehn
\address{
\orgdiv{School of Computation Information and Technology, Department of Mathematics},
\orgname{Technical University of Munich},
\orgaddress{
\street{Boltzmannstra\ss e},
\postcode{85748},
\state{Garching bei München},
\country{Germany}}
}
}

\theoremstyle{thmstyleone}%
\newtheorem{theorem}{Theorem}
\newtheorem{proposition}[theorem]{Proposition}%

\theoremstyle{thmstyletwo}%
\newtheorem{remark}{Remark}%

\theoremstyle{thmstylethree}%

\numberwithin{equation}{section}

\everymath{\displaystyle}
\begin{document}

\DOI{...}
\copyrightyear{2023}
\vol{...}
\pubyear{...}
\access{...}
\appnotes{...}
\copyrightstatement{...}
\firstpage{1}

\abstract{
    Network dynamics is nowadays of extreme relevance to model and analyze complex systems. From a dynamical systems perspective, understanding the local behavior near equilibria is of utmost importance. In particular, equilibria with \emph{at least} one zero eigenvalue play a crucial role in bifurcation analysis. In this paper, we want to shed some light on nilpotent equilibria of network dynamical systems. As a main result, we show that the blow-up technique, which has proven to be extremely useful in understanding degenerate singularities in low-dimensional ordinary differential equations, is also suitable in the framework of network dynamical systems. Most importantly, we show that the blow-up technique preserves the network structure. The further usefulness of the blow-up technique, especially with regard to the desingularization of a nilpotent point, is showcased through several examples including linear diffusive systems, systems with nilpotent internal dynamics, and an adaptive network of Kuramoto oscillators.
}

\keywords{Nilpotent singularities; Network dynamical systems; Blow-up method; Geometric desingularization.}

\maketitle

\section{Introduction}
We consider a networked system of $N$ \emph{scalar} nodes $x_i=x_i(t)\in\R$:
\begin{equation}\label{eq:sys1}
    \begin{split}
        \dot x_i = f_i(x_i,\mu_i)+\sum_{j=1}^Nw_{ij}h_{ij}(x_i,x_j,\lambda_{ij}),\qquad i=1,\ldots,N,
    \end{split}
\end{equation}
where $f_i$ and $h_{ij}$ are smooth functions with some ``internal'' parameters $\mu_i\in\R^{u_i}$, ``interaction parameters'' $\lambda_{ij}\in\R^{l_{ij}}$, and ``weights'' $w_{ij}\in\R$. We refer to $f_i$ as \emph{the internal dynamics} and to $h_{ij}$ as \emph{the interaction}. The topology of the network is encoded in $w_{ij}$, which are the coefficients of a weighted adjacency matrix. In compact form, we can write \eqref{eq:sys1} as
\begin{equation}\label{eq:sys2}
    \dot x = F(x,\mu) + H(x,w,\lambda),
\end{equation}
where $F$ includes all the internal dynamics and $H$ all the interactions; and with $x\in\R^N$, $\mu\in\R^{u_1+\cdots+u_N}=\R^u$, $w\in\R^{N^2}$ and $\lambda\in\R^l$ with $l=\sum_{i,j=1}^N l_{ij}$. We shall say that a system \emph{has a network structure} \cite{newman2018networks,porter2016dynamical,strogatz2001exploring} if it can be (re-)written in the form \eqref{eq:sys1}, or equivalently \eqref{eq:sys2}. The main idea is that one can clearly distinguish the internal, or uncoupled, dynamics from the network interaction. In the rest of this paper, we shall concentrate on studying the dynamics near a nilpotent equilibrium point of \eqref{eq:sys1}. We recall that given a differential equation $\dot x=f(x)$, $x\in\R^n$, an equilibrium point $x^*$ is said to be nilpotent if the eigenvalues of the Jacobian $\textnormal D_xf(x^*)$ are all zero. 

The main conclusion drawn from the analysis presented in this paper is that \emph{the blow-up transformation preserves the network structure}. We point out that, a priori, there is no guarantee for a coordinate change to preserve any sort of structure. This can be assessed by simply attempting a linear transformation on e.g. \eqref{eq:sys2}. The blow-up (see section \ref{sec:bu}) is a singular coordinate transformation that has been extensively used to desingularize the local dynamics of (mostly) low dimensional differential equations. Some further important observations are the following (for the details see section \ref{sec:main_bu}):
\begin{itemize}
    \item[\small$\bullet$ ] A node-directional blow-up induces dynamics on the ``blown-up'' edges. However, these dynamics occur at higher orders. That is, up to leading order terms, the blown-up network is still static, see section \ref{sec:node_bu}.
    \item[\small$\bullet$ ] A parameter-directional blow-up is, essentially, a rescaling such that a distinguished parameter is set to $1$ and the rest are kept small, see section \ref{sec:par_bu}.  
    \item[\small$\bullet$ ]  The blow-up seems to be especially useful for slowly adaptive networks, where, for example, an edge-directional blow-up leads to a static network with a distinguished edge fixed to $1$. Moreover, the adaptation rule that defines the dynamics of such distinguished edge is visible, globally, in the whole network, see section \ref{sec:edge_bu}.
\end{itemize}

Before going into the technical details, let us motivate our interest with some examples.

\subsection{Motivating examples}
In this section, we motivate considering networked systems with nilpotent equilibria through some examples. 


\begin{description}
    \item[Nilpotent internal dynamics: ] Consider a weakly coupled network of the form \eqref{eq:sys1} given by
\begin{equation}
    \begin{split}
        \dot x_i = f_i(x_i,\mu_i)+\ve\sum_{j=1}^Nw_{ij}h_{ij}(x_i,x_j,\lambda_{ij}),\qquad i=1,\ldots,N,
    \end{split}
\end{equation}
where $0<\ve\ll1$, and such that there is a fixed set of parameters $\mu_i^*$, and a point $x^*=(x_1^*,\ldots,x_N^*)$, such that $f_i(x_i^*,\mu_i^*)=\frac{\partial f_i}{\partial x_i}(x_i^*,\mu_i^*)=0$. In this case, the point $x^*$ is nilpotent for $\ve=0$ and an approach to investigate the dynamics for $\ve$ small could be via the blow-up technique, as we shall present in this paper. See a concrete example in section \ref{sec:example_nilpotent_internal}.

\begin{remark}
    Even though we focus on networks of scalar nodes, dynamic networks with nilpotent internal dynamics frequently appear once the nodes are considered multidimensional. For example, one may consider coupled neuronal models \cite{schwemmer2012theory} where each neuron modeled, e.g. via the Hodgkin-Huxley model or its simplifications contains nilpotent points in its internal dynamics. In a more general setting, nilpotent Hopf bifurcations have been considered for coupled cell networks \cite{elmhirst2006nilpotent}, see also \cite{golubitsky2006nonlinear,golubitsky2009bifurcations,maria2006homogeneous}. In such a case, an adapted version of \cite{hayes2016geometric} together with the blow-up analysis considered here could be applicable as well.
\end{remark}


\item[A class of nonlinear consensus protocols: ] Consider the general consensus protocol
\begin{equation}
    \dot x = -L\left(Ax+\sum_{i=2}^p F_i(x)\right),
\end{equation}
where $x\in\R^N$, $p\geq2$, $L$ is a Laplacian matrix, $A$ is a non-singular matrix and $F_i$ is a homogeneous polynomial vector field of degree $i$. The classical consensus protocol is obtained when $A=I$ and $F_i=0$ for all $i\geq2$. On the other hand, if $A=0$ we immediately have that the origin is nilpotent. Furthermore, even if one simply considers $\dot x=-LF_2(x)$, any $x^*$ so that $F(x^*)\in\ker L$ would correspond to a nilpotent point. Since the kernel of $L$ is non-trivial, nilpotent points become relevant in this situation. See \cite{bonetto2022nonlinear} for a particular case study, and \cite{jardon2020fast} for a similar, but adaptive, example.


\item[Canceling linear dynamics: ]  Consider \eqref{eq:sys1} in the particular matrix form
\begin{equation}
    \dot x = Ax+ Bx+\cdots,
\end{equation}
where $A$ is diagonal, accounting for the (leading part of the) internal dynamics, and $B$ is a matrix describing the (leading part of the) interaction; the dots stand for higher-order terms\footnote{Throughout this document, ``higher-order terms'' means terms of order $\cO(|x|^k)$ as $|x|\rightarrow 0$ for some $k>1$ to be specified when appropriate.}. One can then realize that, upon variation of parameters, some components of these linear parts may ``cancel'' in the sense of having a nilpotent form. Clearly, this does not require $A=-B$ in general. For example, consider the homogeneous case where $A=aI_N$ with $a\in\R$. So, it suffices to notice that if $T$ is the (complex) matrix that brings $B$ into its Jordan canonical form, then the transformation $x\mapsto T^{-1}x$ implies $T(aI_N+B)T^{-1}=aI_N+J_B$, where $J_B$ is the Jordan canonical form of $B$, and is upper-triangular. In turn, $aI_N+J_B$ can be nilpotent, or at least have some zero eigenvalues, for $B\neq-A$. For an example see section \ref{sec:diffusive}, and \cite{nijholt2023chaotic} for a related scenario.

\item[(Slowly) adaptive networks: ] As in the classical (low-dimensional) setting, nilpotent singularities are extremely relevant for slow-fast systems \cite{kuehn2015multiple,wechselberger2020geometric}. Consider the (slowly) adaptive dynamic network \cite{berner2023adaptive,sawicki2023perspectives}
\begin{equation}\label{eq:adapt0}
    \begin{split}
        \dot x_i &= f_i(x_i,\mu_i)+\sum_{j=1}^Nw_{ij}h_{ij}(x_i,x_j,\lambda_{ij})\\
        \dot w_{ij} &= \ve g_{ij}(x,w_{ij}).
    \end{split}
\end{equation}
If one were to attempt the analysis of \eqref{eq:adapt0} with techniques derived from Geometric Singular Perturbation Theory (GSPT), the nilpotent points of \eqref{eq:adapt0} restricted to $\ve=0$ define a class of singularities nearby which the solutions may exhibit a quite intricate behavior for $0<\ve\ll 1$. As a particular example consider a family of adaptive Kuramoto oscillators
\begin{equation}
	\begin{split}
		\phi_i' &= \Omega - \frac{1}{N}\sum_{j=1}^N  w_{ij} \sin(\phi_i-\phi_j+\alpha)\\
		 w_{ij} ' &=-\ve\left( \sin(\phi_i-\phi_j+\beta)+ w_{ij}  \right),
	\end{split}
\end{equation}
with parameters $\alpha,\beta$ close to zero.
A phase-locked equilibrium is given by
\begin{equation}
\begin{split}
	\phi_i^* &= \psi_i\\
	w_{ij}^* &= -\sin(\psi_i-\psi_j+\beta)\\
    \Omega &= -\frac{1}{N}\sum_{j=1}^N\sin(\psi_i-\psi_j+\beta)\sin(\psi_i-\psi_j+\alpha),
\end{split}
\end{equation}
for some $\psi_i\in[0,2\pi)$, $i=1,\ldots,N$. As an example, let us take a 4-node network and $\psi=(\psi_1,\psi_2,\psi_3,\psi_4)=\left(0,\frac{\pi}{2},\pi,\frac{3\pi}{2}\right)$. For this choice of equilibrium phases, the phase-locked equilibrium is nilpotent for $\alpha=\beta=0$ and $\Omega=-\frac{1}{2}$. Indeed we have that the relevant part of the linearization, at the aforementioned phase-locked equilibrium, is given by
\begin{equation}
\begin{split}
	J &=
	\frac{1}{4}\begin{bmatrix}
		w_{13} & 0 & -w_{13} & 0\\
		0 & w_{24} & 0 & -w_{24}\\
		-w_{31} & 0 & w_{31} & 0 \\
		0 & -w_{42} & 0 & w_{42}
	\end{bmatrix}.
\end{split}
\end{equation}
So, further noting that $w_{13}^*=w_{31}^*=w_{24}^*=w_{42}^*=0$ we have that $J=0$, implying that the phase-locked equilibrium is nilpotent. This example is further detailed in section \ref{sec:example_adaptive}.

\begin{remark}
    Adaptivity is, of course, not necessary. See for example \cite{sclosa2021completely} where nilpotent equilibria of the classical Kuramoto model are considered.
\end{remark}

\end{description}

As motivated in the previous examples, we are interested in network dynamical systems \eqref{eq:sys1}, equivalently \eqref{eq:sys2}, with a nilpotent singularity. Nilpotent points are highly important because they hint at the possibility of complicated bifurcations occurring in their vicinity. From now on, we shall assume that there is a set of parameters $\mu^*$, $\lambda^*$ and edge weights $w^*$ such that $x^*$ is an \emph{isolated} nilpotent equilibrium point of \eqref{eq:sys2}. 
\begin{remark}
Assuming $x^*$ being isolated is not too strict, at least for our arguments regarding the blow-up transformation preserving the network structure, because one can also blow-up higher-dimensional sets of equilibria in a similar way. We present a brief description of the blow-up transformation and several relevant references in section \ref{sec:bu}.   
\end{remark}

Without loss of generality, we further assume that the nilpotent point is at
$x^*=0$ for $(\mu^*,\lambda^*,\omega^*)=(0,0,0)$. However, we emphasize that, as presented in the examples, a nilpotent point does not necessarily have to be related to a disconnected network.  Let $a_i(\mu)=\frac{\partial f_i}{\partial x_i}(0,\mu)$,  $D(\mu)=\diag\left\{ a_i(\mu)\right\}_{i=1}^N$, and $A(\omega,\lambda)=\frac{\partial H}{\partial x}(0,\omega,\lambda)$. So, we can write \eqref{eq:sys2} as
\begin{equation}\label{eq:sys3}
    \dot x = (D(\mu) + A(\omega,\lambda))x+\tilde F(x,\mu)+\tilde H(x,w,\lambda),
\end{equation}
where, from our assumptions, $D(0) + A(0,0)$ is a nilpotent matrix (and non of which needs to be the zero matrix) and $\tilde F$ and $\tilde H$ stand, respectively, for the higher-order terms (by this we mean monomials of degree higher than one) of the internal dynamics and of the interaction.

It is an important question whether one may indeed expect the leading term $D(\mu)+A(\omega,\lambda)$ in \eqref{eq:sys3} to be nilpotent for parameters $(\mu,\omega,\lambda)=(\mu^*,\omega^*,\lambda^*)$. Roughly speaking, since $n$ eigenvalues are required to be zero, we expect that this can be achieved, generically, for $n$-parameter families of dynamic networks \eqref{eq:sys3}. More formally, since $GL(n)$ has dimension $n^2$ and since there are at most $n(n-1)$ ``different'' nilpotent matrices, one expects that such matrices appear generically whenever $D(\mu)+A(\omega,\lambda)$ has at least $n$ independent parameters. In other words, if $D(\mu)+A(\omega,\lambda)$ has at least $n$ independent parameters, then there is at least one particular choice $(\mu^*,\omega^*,\lambda^*)$ such that $D(\mu^*)+L(\omega^*,\lambda^*)$ is nilpotent. Moreover, since the only nilpotent symmetric matrix is the zero matrix, we shall consider that the interaction is \emph{directed}.

\begin{remark}
    We emphasize that we do not put $D(0) + A(0,0)$ in normal form since, generally speaking, a similarity transformation destroys the network structure of the system.
\end{remark}

\section{Blowing up preserves the network structure}\label{sec:main_bu}

Our aim in this section is to show that the (singular) coordinate transformation known as blow-up, when applied to a nilpotent singularity of a network dynamical system, preserves the network structure. Due to the importance of the blow-up, although the literature contains already plenty of information on it \cite{alvarez2011survey,jardon2019survey,kuehn2015multiple}, we prefer to provide a brief recollection of what is known as a blow-up in the way it is used in this paper. More importantly, in this section, we provide certain terminology and fix some notations that are later used in the examples of section \ref{sec:examples}.

\subsection{The blow-up}\label{sec:bu}
Let $X:\R^n\to \R^n$ be a smooth vector field\footnote{What follows also holds for smooth vector fields on smooth manifolds by taking a chart in a neighborhood of the nilpotent point.} such that $X(0)=0$ and $0\in\R^n$ is a nilpotent point. Let $\phi:\S^{n-1}\times I\to \R^n$ be a weighted polar transformation given by
\begin{equation}
    (\bx_1,\ldots,\bx_n,r)\mapsto(r^{\alpha_1}\bx_1,\ldots,r^{\alpha_n}\bx_n)=(x_1,\ldots,x_n),
\end{equation}
where $\alpha_i\in\N$ for all $i=1,\ldots,n$, $\sum_{i=1}^n\bx_i^2=1$ (i.e. $(\bx_1,\ldots,\bx_n)\in\S^{n-1}$) and $r\in I\subset\R$ where $I$ is an interval containing the origin. We notice that $\phi^{-1}$ maps the origin $0\in\R^n$ to the sphere $\S^{n-1}\times\left\{ 0 \right\}$ and we usually say that ``the origin is blown-up to a sphere''. On the other hand, $\phi$ is a diffeomorphism whenever $r>0$. Due to the weights in the transformation, one refers to the transformation as quasihomogeneous (and homogeneous when all weights are the same). These weights often facilitate the desingularization (see a description below), but their appropriate choice can be quite complicated. If the diagram of figure \ref{fig:bu_diag} commutes, then we call $\bX$ the blown-up vector field.
\begin{figure}[htbp]
    \centering
    \begin{tikzpicture}
        \node (a) at (0,0) {$\S^{n-1}\times I$};
        \node (b) at (3,0) {$\R^n$};
        \node (c) at (0,-2) {$T\S^{n-1}\times I$};
        \node (d) at (3,-2) {$\R^n$};
        \draw[->] (a) -- (b) node[midway, above]{$\phi$};
        \draw[->] (c) -- (d) node[midway, above]{$\DD\phi$};
        \draw[->] (a) -- (c) node[midway, left]{$\bX$};
        \draw[->] (b) -- (d) node[midway, right]{$X$};
    \end{tikzpicture}
    \caption{Commutative diagram defining a blow-up transformation.}
    \label{fig:bu_diag}
\end{figure}

Since $X(0)=0$, it follows that $\bX|_{r=0}=0$. However, if the weights are appropriately chosen, one can often \emph{desingularize} $\bX$ by dividing by some power of $r$. In other words, one can define the \emph{desingularized} vector field $\tilde X=\frac{1}{r^k}\bX$ such that $\tilde X(0)\neq0$ and $\tilde X$ has semi-hyperbolic singularities only (and no more nilpotent ones).  Recalling that $\phi$ is a diffeomorphism for $r>0$, the main advantage now is that for $r>0$ small (i.e. near $\S^{n-1}\times\left\{ 0 \right\}$) 
 one can recover the dynamics of $\bX$ from those of $\tilde X$, and in turn these describe the dynamics of $X$ near the origin. In practice, however, working with spherical coordinates can become cumbersome very quickly. Hence, one prefers to work with \emph{directional} blow-ups, which are defined by replacing $\S^{n-1}$ by its charts. Moreover, blowing up using a sphere is not necessary, as one can blow-up, for example, to a hyperbolic manifold \cite{kuehn2016remark}.

Although the origins of the blow-up transformation are in algebraic geometry, the blow-up transformation has proven extremely useful in dynamical systems to study nilpotent singularities. It has been, for example, used to study degenerate bifurcation problems and their unfoldings \cite{takens1974singularities,dumortier1993techniques}, see also the survey \cite{alvarez2011survey}; or, of particular relevance for this paper, to study some non-hyperbolic points of slow-fast systems \cite{dumortier1996canard,krupa2001extending}, see also the survey \cite{jardon2019survey}. Indeed, in the aforementioned context, one can consider nowadays that the blow-up transformation is a standard tool of Geometric Singular Perturbation Theory \cite{kuehn2015multiple,wechselberger2020geometric}. Over the past decades, the blow-up transformation has been mostly used for low-dimensional problems. In fact, it is known that analytic vector fields up to dimension three can be desingularized after a finite number of blow-ups \cite{panazzolo2002desingularization,panazzolo2006resolution}. In this paper, we will see the usefulness of the blow-up transformation on systems that are usually considered large-dimensional, dynamic networks. As dynamic networks are, usually, of large dimension, a relevant direction in this regard is the blow-up in the context of PDEs \cite{hummel2022geometric,jelbart2023formal,engel2022geometric,engel2020blow}, which so far has been considered without networks structure and, in some cases, can be regarded as mean-field limits of dynamic networks \cite{lancellotti2005vlasov,gkogkas2022graphop}.

\subsection{Structure preservation}
This section is dedicated to showing that a general directional blow-up preserves the network structure of a given dynamical system. 
\begin{remark}
    As described in the previous section, usually one expects that after desingularization via blow-up one gains a certain degree of hyperbolicity. The number of blow-up transformations to completely desingularize a nilpotent singularity is, however, not known a priori. This means that, after a blow-up, some of the new singularities may be semi-hyperbolic or (even still) nilpotent. This is nevertheless already good progress. In the former case, further analysis can be restricted to a lower dimensional subset, the corresponding center manifold, where the blow-up can be applied again. In the latter, the process of blowing-up can also be repeated in a system that is, usually, less degenerate than the starting one. Nevertheless, one expects that after a finite number of blow-ups, one can fully describe the dynamics of a system near a nilpotent singularity.
    
    We emphasize that in this section, we only focus on structure preservation via blowing up. The question of (full) desingularization is a nontrivial one and is highly dependent on the particular problem at hand, and the choice of weights in the quasihomogeneous blow-up (a nontrivial task in itself). We discuss some cases where one can achieve desingularization later in the examples of section \ref{sec:examples}.
\end{remark}

Using the notation introduced above, let us consider a vector field of the form
\begin{equation}\label{eq:X}
    X:\left\{\begin{aligned}
        \dot x &= F(x,\sigma)+H(x,\sigma)=(D_\sigma+A_\sigma)x+\tilde F(x,\sigma)+\tilde H(x,\sigma)\\
        \dot\sigma &=0,
    \end{aligned}\right.
\end{equation}
where $F$ stands for the internal dynamics, $H$ the interaction, and $\sigma\in\R^p$ includes all possible parameters, hence $D_0+A_0$ is nilpotent. \emph{Our objective is to show that a local vector field obtained by a directional blow-up has the same identifiable structure of internal dynamics plus interaction}.

\begin{remark}
    The blow-up technique is local. Hence, without further mentioning it, we assume that the vector fields we are considering are polynomial.
\end{remark}

Let us define a quasihomogeneous blow-up via the map $\Xi:\R_{\geq0}\times\mathbb S^{N+p-1}\to\R^{N+p}$ given by:
\begin{equation}
    \Xi:(r,\bx_1,\ldots,\bx_n,\bs_1,\ldots,\bs_p)\mapsto\underbrace{(r^{\alpha_1}\bx_1,\dots,r^{\alpha_n}\bx_n,r^{\beta_1}\bs_1,\ldots,r^{\beta_p}\bs_p)}_{=(x_1,\ldots,x_N,\sigma_1,\ldots,\sigma_p)},
\end{equation}
and where
\begin{equation}
    \sum_{i=1}^N\bx_i^2+\sum_{j=1}^p\bs_j^2=1.
\end{equation}

We can define a \emph{directional} blow-up by fixing one of the blown-up variables in $\mathbb S^{N+p-1}$ to $\pm1$ and let the rest be coordinates in $\R^{n-1}$ (this is, of course, reminiscent of a stereographic projection). The sign is, for now, rather inessential, so let us consider charts
\begin{equation}
    \begin{split}
        K_{i}&=\left\{ \bx_i=1 \right\},\qquad i=1,\ldots,N\\
        Q_{j}&=\left\{ \bs_j=1 \right\}, \qquad j=1,\ldots,p.
    \end{split}
\end{equation}

Let us denote a corresponding directional blow-up by
\begin{equation}
    \begin{split}
        \Phi_{i}&=\Xi|_{\bx_i=1},\\
        \Psi_{j}&=\Xi|_{\bs_j=1}.
    \end{split}
\end{equation}

The blown-up vector field in the chart $K_i$ (resp. $Q_j$) is then obtained via the application of the corresponding change of coordinates $\Phi_i$ (resp. $\Psi_j$). For convenience, let
\begin{equation}\label{eq:directional_vfs}
\begin{split}
    \bar X_i &=(\DD\Phi_{i})^{-1}X\circ \Phi_i, \qquad i=1,\ldots,N\\
    \bar Y_j &= (\DD\Psi_{j})^{-1}X\circ \Psi_j,\qquad j=1,\ldots,p.
\end{split}
\end{equation}

\begin{remark}\leavevmode
\begin{itemize}
    \item[\small$\bullet$ ]  The quasi-homogeneous blow-up requires a careful choice of weights so that the local vector fields are well defined. In this section, however, we are only interested in the form the blown-up vector fields take and hence assume that \eqref{eq:directional_vfs} are all well defined, especially at $r=0$ (possibly after a division by some power of $r$).
    \item[\small$\bullet$ ] $\Phi_i$ corresponds to a blow-up of the node $x_i$. Hence, within the context of networks, we refer to it as a ``node blow-up''. On the other hand, the need to blow-up in the direction of a parameter arises when such a parameter is responsible for some singular behavior, for example when it induces a nilpotent singularity. So, if for example, $\sigma_k$ corresponds to some ``weight parameter'' $w_{ij}$, we shall refer to $\Psi_k$ as an ``edge blow-up''. Of course, everything that we will present below can be extended to network dynamical systems given by
    \begin{equation}
    \begin{split}
        \dot x_i &= f_i+\sum_{j=1}^Nw_{ij}h_{ij},\\
        \dot w_{ij} &= g_{ij},
    \end{split}
\end{equation}
where the ``edge-directional blow-up'' would be more apparent. This consideration is, however, inessential for the purposes of the paper, which is to show the preservation of the network structure after the blow-up. See more details in section \ref{sec:edge_bu}.
    
\end{itemize}
\end{remark}

Since the compositions $X\circ\Phi_i$ and $X\circ\Psi_j$ do not change the network structure, we only need to look at the effect of pre-multiplication by the inverse of the derivative of the directional blow-up. 

\subsection{Node directional blow-up}\label{sec:node_bu}

We first focus on the node blow-up, that is on the $\bar X_i$'s. For convenience and simplicity of the exposition, for each $i=1,\ldots,N$ let us reorder the components of $X$ according to the coordinates $$(x_i,x_1,\ldots,x_{i-1},x_{i+1},\ldots,x_N,\sigma_1,\ldots,\sigma_p).$$
We call such a vector filed $X_i$, but notice that we are simply swapping the position of the $i$-th component\footnote{If \eqref{eq:X} is given by $X=\sum_{i=1}^Nf_i\frac{\partial}{\partial x_i}+\sum_{j=1}^p0\frac{\partial}{\partial \sigma_j}$, then
$X_i=f_i\frac{\partial}{\partial x_i}+\sum_{k=1,k\neq i}^Nf_k\frac{\partial}{\partial x_k}+\sum_{j=1}^p0\frac{\partial}{\partial \sigma_j}.$}. In this way the corresponding directional blow-up 
$\Phi_i$ is given in coordinate form by 
\begin{equation}
\begin{pmatrix}
r^{\alpha_i}\\
r^{\alpha_1}\bx_1\\
\vdots\\
r^{\alpha_{i-1}}\bx_{i-1}\\
r^{\alpha_{i+1}}\bx_{i+1}\\
\vdots\\
r^{\alpha_N}\bx_N\\
r^{\beta_1}\bs_1\\
\vdots\\
r^{\beta_p}\bs_p
\end{pmatrix} 
=
\begin{pmatrix}x_i\\x_1\\ \vdots\\x_{i-1}\\x_{i+1}\\ \vdots\\x_N\\ \sigma_1\\ \vdots\\ \sigma_p   
\end{pmatrix},    
\end{equation}
and therefore
\begin{equation}
    \DD\Phi_i=\begin{bmatrix}
        \rho_i & 0_{1\times (N-1)}  & 0_{1\times p}\\
        \chi_i & R_i & 0_{(N-1)\times p}\\
        \zeta & 0_{p\times(N-1)} & B
    \end{bmatrix}
\end{equation}
where 
\begin{equation}
    \begin{split}
        \rho_i&=\alpha_ir^{\alpha_i-1},\\
        \chi_i&=(\alpha_1r^{\alpha_1-1}\bx_1,\cdots,\alpha_{i-1}r^{\alpha_{i-1}-1}\bx_{i-1},\alpha_{i+1}r^{\alpha_{i+1}-1}\bx_{i+1},\cdots,\alpha_{N}r^{\alpha_{N}-1}\bx_{N})^\top,\\
        R_i &=\diag\left\{ r^{\alpha_1},\cdots,r^{\alpha_{i-1}},r^{\alpha_{i+1}},\cdots,r^{\alpha_N}\right\},\\
        \zeta &=(\beta_1 r^{\beta_1-1},\cdots,\beta_pr^{\beta_p-1}\bs_p)^\top,\\
        B&=\diag\left\{ r^{\beta_1},\cdots,r^{\beta_p} \right\}.
    \end{split}
\end{equation}

Since $\DD\Phi_i$ is lower triangular, its inverse is also lower triangular. In fact, we have:
\begin{equation}
    (\DD\Phi)^{-1}=\frac{1}{\rho_i}\begin{bmatrix}
         1 & 0_{1\times(N-1)} & 0_{1\times p} \\
        -R_i^{-1}\chi_i & \rho_i R_i^{-1} & 0_{(N-1)\times p}\\
        -B^{-1}\zeta & 0_{p\times(N-1)} & \rho_iB^{-1}.
    \end{bmatrix}
\end{equation}
\begin{remark}
    As expected, $\DD\Phi$ is invertible for $r\neq0$.
\end{remark}

Thus, the local vector field $\bX_i=(\DD\Phi_i)^{-1}X_i\circ\Phi_i$ can be written in local coordinates\footnote{We recycle the coordinate notation in each chart to avoid introducing obfuscating new terminology. So, for example, the local coordinates in the chart $K_i$ are $(r,\bx_1,\ldots,\bx_{i-1},\bx_{i+1},\ldots,\bx_N,\bs_1,\ldots,\bs_p)$, and similarly the local coordinates in the chart $Q_j$ are $(r,\bx_1,\ldots,\bx_N,\bs_1,\ldots,\bs_{j-1},\bs_{j+1},\ldots,\bs_p)$. The distinction between these local coordinates is only necessary when transitioning from one chart to another. Furthermore, along the text we use the common notation $r^\alpha\bx=(r^{\alpha_1}\bx_1,\ldots,r^{\alpha_N}\bx_N)$, which within the chart $K_i$ should be understood as $r^\alpha\bx=(r^{\alpha_i},r^{\alpha_1}\bx_1,\ldots,r^{\alpha_{i-1}}\bx_{i-1},r^{\alpha_{i+1}}\bx_{i+1},r^{\alpha_N}\bx_N)$, and similarly for $r^\beta\bs$.} as
\begin{equation}
    \bX_i:\begin{cases}
        r'&= \frac{r^{1-\alpha_i}}{\alpha_i}\left(\bar f_i +\bar h_i \right), \\
        \bx_k'&= \frac{1}{r^{\alpha_k}}\left( \bar f_k+\bar h_k-\frac{\alpha_k}{\alpha_i}r^{\alpha_k-\alpha_i}\bx_k(\bar f_i+\bar h_i) \right),\\
        \bs_j'&=-\frac{\beta_j}{\alpha_i}r^{-\alpha_i}\bs_j(\bar f_i+\bar h_i),
    \end{cases}
\end{equation}
where $k=1,\ldots,i-1,i+1,\ldots,N$, $j=1,\ldots,p$ and where $\bar f_\ell$ and $\bar h_\ell$ are the $\ell$-th component of $F\circ\Phi_i$ and $H\circ\Phi_i$ respectively. Notice already that the equations for $\bx_k$ keep, in some sense to be detailed below, the network structure.

Let $[M]_{ij}$ denote the $i,j$ entry of the matrix $M$. Recall from \eqref{eq:sys3} that $F+H=(D_\sigma+A_\sigma)x+\cdots$. This implies that, up to leading order terms, we have:
\begin{equation}
    \begin{split}
        \bar f_i+\bar h_i &= (a_i(r^\beta\bs)+[A_{r^\beta\bs}]_{ii})r^{\alpha_i}+\sum_{\substack{j=1\\j\neq i}}^N[A_{r^\beta\bs}]_{ij}r^{\alpha_j}\bx_j+\cdots\\
        \bar f_k+\bar h_k &= [A_{r^\beta\bs}]_{ki}r^{\alpha_i}+ (a_k(r^\beta\bs)+[A_{r^\beta\bs}]_{kk})r^{\alpha_k}\bx_k+\sum_{\substack{j=1\\j\neq i\\j\neq k}}^N[A_{r^\beta\bs}]_{kj}r^{\alpha_j}\bx_j+\cdots,
    \end{split}
\end{equation}
which hints to the choice of the blow-up weights $\alpha_i=\alpha$ for all $i=1,\dots,N$ so that every local vector field $\bX_i$ is well defined for $r=0$. Under such a choice we have, up to leading order terms, the blown-up vector field $\bX_i$ given by:
\begin{equation}\label{eq:bu1}
    \begin{split}
        r'&= \frac{r}{\alpha}\left( (a_i(r^\beta\bs)+[A_{r^\beta\bs}]_{ii})+\sum_{\substack{j=1\\j\neq i}}^N[A_{r^\beta\bs}]_{ij}\bx_j+\cdots\right)\\
        \bx_k'&=\left( [A_{r^\beta\bs}]_{ki}+ (a_k(r^\beta\bs)+[A_{r^\beta\bs}]_{kk})\bx_k+\sum_{\substack{j=1\\j\neq i\\j\neq k}}^N[A_{r^\beta\bs}]_{kj}\bx_j+\cdots \right)\\
        &\qquad -\bx_k\left( (a_i(r^\beta\bs)+[A_{r^\beta\bs}]_{ii})+\sum_{\substack{j=1\\j\neq i}}^N[A_{r^\beta\bs}]_{ij}\bx_j+\cdots\right)\\
        \bs_j'&=-\frac{\beta_j}{\alpha}\bs_j\left( (a_i(r^\beta\bs)+[A_{r^\beta\bs}]_{ii})+\sum_{\substack{j=1\\j\neq i}}^N[A_{r^\beta\bs}]_{ij}\bx_j+\cdots\right).
    \end{split}
\end{equation}

For $r=0$ we have that $r'=0$ and
\begin{equation}\label{eq:bu1r0}
    \begin{split}
        \bx_k'&= [A_0]_{ki} + (a_k(0)+[A_0]_{kk}-a_i(0)-[A_0]_{ii})\bx_k+\sum_{\substack{j=1\\j\neq i\\j\neq k}}^N[A_{0}]_{kj}\bx_j+\cO(|\bx|^2)\\
        \bs_j'&= -\frac{\beta_j}{\alpha}\left( a_i(0)+[A_0]_{ii}\right)\bs_j+\cO(\bs_j\bx),
    \end{split}
\end{equation}
where $k=1,\ldots,i-1,i+1,\ldots,N$. We can see that the blown-up system \eqref{eq:bu1} has ``dynamic" parameters $\bs$. However, for $r=0$ (i.e. \eqref{eq:bu1r0}) the network is static.  Moreover, we see that the equilibrium point of $\bx'$ may not be at the origin any more.

\begin{description}
\item[Interpretation: ] one can give several network interpretations to the equation for $\bx'$ in \eqref{eq:bu1r0}. For this, it is convenient to rewrite the first part of \eqref{eq:bu1r0} as
\begin{equation}
    \bx_k'= [A_0]_{ki} + (a_k(0)-a_i(0)-[A_0]_{ii})\bx_k+\sum_{\substack{j=1\\j\neq i}}^N[A_{0}]_{kj}\bx_j+\cdots,
\end{equation}
where we can notice a network with $N-1$ nodes $(\bx_1,\ldots,\bx_{i-1},\bx_{i+1},\ldots,\bx_N)$.
\begin{itemize}
    \item[\small$\bullet$ ] As a first interpretation, one may say that the internal dynamics are now modified (``shifted'' by $[A_0]_{ki}$ and ``scaled'' by $[A_0]_{ii}$) by the influence of node $i$ (which has been blown-up and ``does not exists anymore''). Regarding the interaction, that remains the same, except that the node $i$ does not appear anymore.

    \item[\small$\bullet$ ]  Another interpretation could be given when we re-write \eqref{eq:bu1r0} as:
\begin{equation}
    \bx_k' = [ A_0]_{ki} + a_k(0)\bx_k + \sum_{\substack{j=1\\j\neq i}}^N [\tilde A_0]_{kj} \bx_j+\cdots,
\end{equation}
where $[\tilde A_0]_{kj}=-a_i(0)-[L_0]_{ii}+[A_0]_{kk}$ if $j=k$ and $[\tilde A_0]_{kj}=[A_0]_{kj}$ otherwise. This, instead can be interpreted as the ``self-interaction'' having been modified by the (nonexistent) node $i$, while the term $[A_0]_{ki}$ can be regarded as a constant input, which is nonzero if and only if there is a connection from node $i$ to node $k$.
\end{itemize}

$ $


Finally, if $r=0$ is hyperbolic in the $r$-direction (which would be the case under desingularization), then we expect that \eqref{eq:bu1} is a regular perturbation of \eqref{eq:bu1r0} for $r>0$ sufficiently small. This means that the dynamics of  \eqref{eq:X} in a small neighborhood of the nilpotent origin can be determined from to those of \eqref{eq:bu1r0}, see a relevant example in  section \ref{sec:examples}.
\end{description}

\subsection{Parameter directional blow-up}\label{sec:par_bu}
We now turn our attention to the blow-up in the parameter direction (the $\bar{Y}_i$'s in \eqref{eq:directional_vfs}). For convenience, let us recall that we are considering the (extended) vector field
\begin{equation}\label{eq:bup1}
    X:\left\{\begin{aligned}
        \dot x &= F(x,\sigma)+H(x,\sigma)=(D_\sigma+L_\sigma)x+\tilde F(x,\sigma)+\tilde H(x,\sigma)\\
        \dot\sigma &=0,
    \end{aligned}\right.
\end{equation}
where $L_0+D_0$ is nilpotent. If one would rewrite \eqref{eq:bup1} as $X=\sum_{i=1}^Nf_i\frac{\partial}{\partial x_i}+\sum_{j=1}^p0\frac{\partial}{\partial \sigma_j}$, it is rather convenient to consider the vector field 
\begin{equation}
    Y_j=0\frac{\partial}{\partial \sigma_j}+\sum_{i=1}^Nf_i\frac{\partial}{\partial x_i}+\sum_{k=1,\,k\neq j}^p0\frac{\partial}{\partial \sigma_k},
\end{equation}
with the directional blow-up $\Psi_j$ given by local coordinates:
\begin{equation}
\begin{pmatrix}
r^{\beta_j}\\
r^{\alpha}\bx_1\\
\vdots\\
r^{\alpha}\bx_N\\
r^{\beta_1}\bs_1\\
\vdots\\
r^{\beta_{j-1}}\bs_{j-1}\\
r^{\beta_{j+1}}\bs_{j+1}\\
\vdots\\
r^{\beta_p}\bs_p
\end{pmatrix} 
=
\begin{pmatrix}
\sigma_j\\
x_1\\ \vdots\\x_N\\ \sigma_1\\ \vdots\\ \sigma_{j-1}\\  \sigma_{j+1}\\
\vdots\\
\sigma_p   
\end{pmatrix}.    
\end{equation}
Recall that from the previous section we have already chosen the blow-up weights $\alpha_i=\alpha$ for all $i=1,\ldots,N$. Therefore
\begin{equation}
    \DD\Psi_j=\begin{bmatrix}
        \rho_j & 0_{1\times N} & 0_{1\times (p-1)}\\
        \chi & R & 0_{N\times(p-1)}\\
        \zeta_j & 0_{(p-1)\times N} & B_j
    \end{bmatrix},
\end{equation}
where 
\begin{equation}
    \begin{split}
        \rho_j &=\beta_j\bs_j^{\beta_j-1}\\
        \chi &= (\alpha r^{\alpha-1}\bx_1,\ldots,\alpha r^{\alpha-1}\bx_N)^\top\\
        R &=\diag\left\{ r^{\alpha},\ldots,r^{\alpha}\right\}\\
        \zeta_j&=(\beta_1r^{\beta_1-1}\bs_1,\ldots,\beta_{j-1}r^{\beta_{j-1}-1}\bs_{j-1},\beta_{j+1}r^{\beta_{j+1}-1}\bs_{j+1},\ldots,\beta_pr^{\beta_p-1}\bs_p)^\top\\
        B_j &=\diag\left\{ r^{\beta_1},\ldots,r^{\beta_{j-1}},r^{\beta_{j+1}},\ldots,r^{\beta_{p}}\right\}.
    \end{split}
\end{equation}
Following similar steps as in the previous section, we find that the corresponding local vector field (in the chart $Q_j$) reads as
\begin{equation}
    \bar Y_j=\begin{cases}
        r' &=0\\
        \bx_i'&=\frac{1}{r^\alpha}(\bar f_i+\bar h_i)\\
        \bs_k'&=0,
    \end{cases}
\end{equation}
where $k=1,\ldots,j-1,j+1,\ldots,p$, and $\bar f_\ell$ and $\bar h_\ell$ are the $\ell$-th component of $F\circ\Psi_j$ and $H\circ\Psi_j$ respectively. Again, it is evident that the blown-up vector field $\bar Y_j$ still has a network structure. More specifically we have:
\begin{equation}
    \begin{split}
        \bar f_i&=a_i(r^\beta\bs)r^\alpha\bx_i + \cO(r^{2\alpha})\\
        \bar h_i &= \sum_{k=1}^N[A_{r^\beta\bs}]r^\alpha\bx_k+\cO(r^{2\alpha}).
    \end{split}
\end{equation}
We recall that in the chart $Q_j$ we use the short-hand notation $$r^\beta\bs=(r^{\beta_1}\bs_1,\ldots,r^{\beta_{j-1}}\bs_{j-1},r^{\beta_j},r^{\beta_{j+1}}\bs_{j+1},\ldots,r^{\beta_p}\bs_p).$$
\begin{description}
    \item[Interpretation: ] in this case the blow-up is essentially a rescaling, where one wants to focus on the influence of the particular parameter $\sigma_j$. After successful desingularization (which depends on the vector fields and the choice of blow-up weights), the blown-up vector field $\bar Y_j$ is, qualitatively speaking, $Y_j$ with the parameter $\bs_j=1$ (or $\bs_j=-1$ depending on the direction of the blow-up). See the example in section \ref{sec:example_adaptive}. 
\end{description}

From the above insight, since a parameter $\sigma_i\in\sigma$ could be an edge weight, and due to its relevance in applied sciences, we now discuss the edge-directional blow-up in the context of adaptive networks.

\subsection{Edge-directional blow-up for adaptive networks}\label{sec:edge_bu}
Let us now consider a slowly adaptive network
\begin{equation}\label{eq:model_adap}
    \begin{split}
        \dot x_i &= f_i(x_i)+\sum_{j=1}^Nw_{ij}h_{ij}(x_i,x_j)\\
        \dot w_{ij} &= \ve g_{ij}(x,w_{ij}).
    \end{split}
\end{equation}
\begin{remark}
    Frequently, the adaptation rule $g_{ij}$ in applications depends only on the nodes $(i,j)$, see e.g. \cite{berner2023adaptive}; but for our purposes such specification is not relevant.
\end{remark}
Consider an edge-directional blow-up for the edge $(k,l)$ (for simplicity we only consider the positive direction)
\begin{equation}\label{eq:adapt_bu}
    x_i=r\bx_i, \, w_{ij}=r\bw_{ij}, \, \ve=r\be,\, w_{kl}=r,\qquad ij\neq kl.
\end{equation}
\begin{remark}
    In this section, we only focus on the preservation of network structure and its interpretation. If more details are known from the particular functions in \eqref{eq:model_adap}, then a more appropriate choice of blow-up \eqref{eq:adapt_bu} can be made.
\end{remark}
The corresponding blown-up system reads as:
\begin{equation}\label{eq:ad_bu1}
    \begin{split}
        \dot r &= r\be \bar g_{kl}(r,\bx)\\
        \dot\bx_i &= \frac{\bar f_i(r,\bx_i)}{r}+\sum_{j=1}^N\bw_{ij}\bar h_{ij}(r,\bx_i,\bx_j)-\be\bx_i\bar g_{kl}(r,\bx),\qquad i\neq k\\
        \dot\bx_k &= \frac{\bar f_k(r,\bx_k)}{r}+\sum_{j=1,j\neq l}^N\bw_{kj}\bar h_{kj}(r,\bx_k,\bx_j)+\bar h_{kl}(r,\bx_k,\bx_l)-\be\bx_k\bar g_{kl}(r,\bx),\\
        \dot \bw_{ij} &= \be(\bar g_{ij}(r,\bx,\bar w_{ij})-\bar w_{ij}\bar g_{kl}(r,\bx)), \quad ij\neq kl\\
        \dot\be &=-\be^2\bar g_{kl}(r,\bx)
    \end{split}
\end{equation}

The dynamics restricted to the invariant subset $\left\{ \be=0 \right\}$ are often useful. One accordingly has that \eqref{eq:ad_bu1} reduces to
\begin{equation}\label{eq:ad_bu2}
    \begin{split}
        \dot r &= 0\\
        \dot\bx_i &= \frac{\bar f_i(r,\bx_i)}{r}+\sum_{j=1}^N\bw_{ij}\bar h_{ij}(r,\bx_i,\bx_j),\qquad i\neq k\\
        \dot\bx_k &= \frac{\bar f_k(r,\bx_k)}{r}+\sum_{j=1,j\neq l}^N\bw_{kj}\bar h_{kj}(r,\bx_k,\bx_j)+\bar h_{kl}(r,\bx_k,\bx_l),\\
        \dot \bw_{ij} &= 0, \quad ij\neq kl\\
    \end{split}
\end{equation}
which is the model of a static network, with weights $\bar w_{ij}$ and fixed weight $\bar w_{kl}=1$. If $\be$ remains small, then \eqref{eq:ad_bu1} can be interpreted as a small perturbation of the static model \eqref{eq:ad_bu2}. In this case, we notice that the weights $\bar w_{ij}$, $ij\neq kl$, would evolve slowly, while $\bar w_{kl}$ is still fixed to $1$. Moreover, we notice that the adaptation effects of $w_{kl}$ are ``globally transported'' to the dynamics of the nodes. This may have important applications to elucidate the influence of certain edges on the rest of the network. If \eqref{eq:ad_bu2} has been desingularized, then the perturbation \eqref{eq:ad_bu1} would be regular.

\section{Examples}\label{sec:examples}

In this section, we present several examples that highlight the use of the blow-up for network dynamical systems. In particular, we focus on network structure preservation and the possibility of desingularization after blow-up.

\subsection{Linearly Diffusive systems}\label{sec:diffusive}

Our first example concerns a set of diffusively coupled nodes with linear internal dynamics: 
\begin{equation}\label{eq:ex_diff_m}
    \dot x_i = a_ix_i+\sum_{j=1}^Nw_{ij}(x_j-x_i),
\end{equation}
with parameters $a_i\in\R$, $a_i\neq0$, and $w_{ij}\in\R$ such that the origin is nilpotent. For the moment, we focus on describing the flow near the origin \emph{with the parameters fixed}. 

To start, let us consider the case where $N=2$, that is:
\begin{equation}\label{eq:ex_diff_eq0}
    \begin{split}
        \dot x_1&=a_1x_1+w_{12}(x_2-x_1)\\
        \dot x_2&=a_2x_2+w_{21}(x_1-x_2).
    \end{split}
\end{equation}
For this, the nilpotency condition is given by the two simultaneous equations
\begin{equation}\label{eq:diff_nil}
\begin{split}
    a_1-w_{12}+a_2-w_{21} &=0\\
    a_1a_2-a_1w_{21}-a_2w_{12} &=0.
\end{split}
\end{equation}
Notice that for these equations to have a nontrivial solution it is necessary that $w_{12}\neq w_{21}$ and $a_1\neq a_2$ which is fulfilled, generically, by having (at least) two independent parameters. Assuming that $a_1,a_2$ are independent, we have that $w_{ij}=w_{ij}^*:=\frac{a_i^2}{a_i-a_j}$, makes the origin nilpotent.

Consider the node-directional blow-up\footnote{It can be easily checked that these two charts are sufficient as in the charts $K_2^\pm=\left\{ \bx_2=\pm1 \right\}$ one obtains analogous local vector fields.}
\begin{equation}
    x_1=\pm r, \; x_2= r\bx_2,
\end{equation}
where the $\pm$ sign accounts for the charts $K_1^\pm:\left\{ \bx_1=\pm1 \right\}$, and which leads to the local vector field
\begin{equation}
    \begin{split}
        r'&=r\left( a_1\pm w_{12}(\bx_2\mp 1)\right)\\
        \bx_2'&=\pm w_{21}+\bx_2(-a_1+a_2+w_{12}-w_{21}\mp w_{12}\bx_2),
    \end{split}
\end{equation}
whereby substituting $w_{ij}=w_{ij}^*=\frac{a_i^2}{a_i-a_j}$ we get
\begin{equation}\label{eq:ex_diff_bu1}
    \begin{split}
        r' &= -\frac{a_1r}{a_1-a_2}(a_2\mp a_1\bx_2)\\
        \bx_2' &=\mp\frac{(a_2\mp a_1\bx_2)^2}{a_1-a_2}.
    \end{split}
\end{equation}
The equilibrium point of \eqref{eq:ex_diff_bu1}, namely $(r,x_2)=(0,\pm\frac{a_2}{a_1})$, is still nilpotent, and in fact the linearization matrix at the equilibrium point is the zero matrix. However, as a slight advantage, \eqref{eq:ex_diff_bu1} is in triangular form and can be desingularized\footnote{The desingularized vector is found by dividing \eqref{eq:ex_diff_bu1} by $(a_2\mp a_1\bx_2)$, and provides the phase-portrait of \eqref{eq:ex_diff_bu1} away from $(a_2\mp a_1\bx_2)=0$ after reversing the direction of the flow in the region $(a_2\mp a_1\bx_2)<0$.} to find the classification shown in figure \ref{fig:ex_diff_class}.
 \begin{figure}[htbp]
     \centering
     \begin{tikzpicture}
         \node at (0,0){
\includegraphics[]{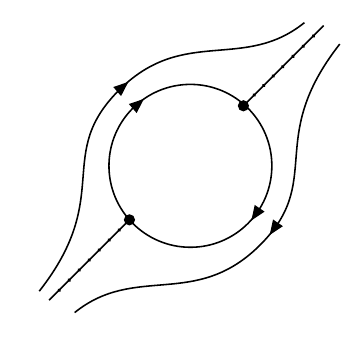}
         };
          \node at (7,0){
\includegraphics[]{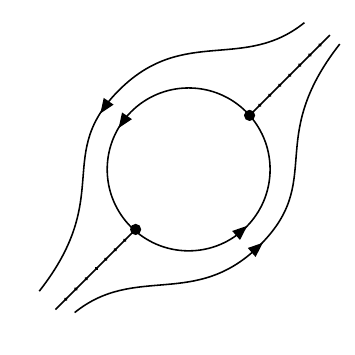}
         };
         \node at (0,-8.5){$\frac{a_1}{a_1-a_2}>0$};
         \node at (7,-8.5){$\frac{a_1}{a_1-a_2}<0$};
         \node at (0,-5){
         \includegraphics[]{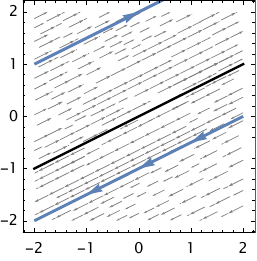}
         };
         \node at (7,-5){
         \includegraphics[]{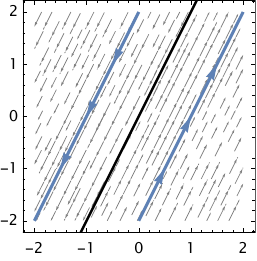}
         };
         \node at (0,-7.5){$x_1$};
         \node at (7,-7.5){$x_1$};
         \node at (-2.5,-5.0){$x_2$};
         \node at (4.5,-5.0){$x_2$};
     \end{tikzpicture}
     \caption{Classification of the flow of the blown-up system \eqref{eq:ex_diff_bu1}; left for $\frac{a_1}{a_1-a_2}>0$ and right for $\frac{a_1}{a_1-a_2}<0$. We show in the first row the global phase-portrait which is obtained by gluing together the flow on the charts $K_1^\pm$. The dotted line corresponds to the line of equilibria $x_2=\frac{a_2}{a_1}x_1$. The shown line is a representative of the case $a_1>0$, $a_2>0$. In the second row, we show a representative simulation of the original system \eqref{eq:ex_diff_eq0} with $w_{ij}=w_{ij}^*=\frac{a_i^2}{a_i-a_j}$. For these simulations we use $a_1=2$, $a_2=1$ on the left; $a_1=1$, $a_2=2$ on the right.}
     \label{fig:ex_diff_class}
 \end{figure}
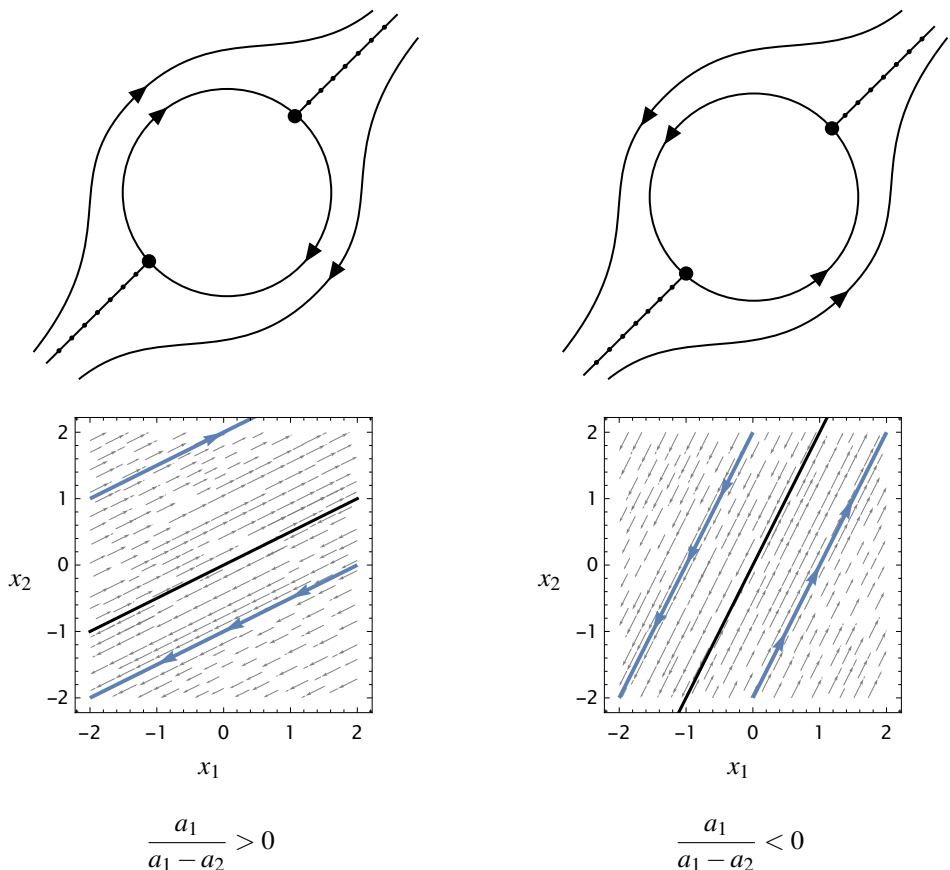

We further emphasize that the equilibrium of  \eqref{eq:ex_diff_bu1} is as degenerate as it can get, in the sense that shifting the equilibrium point to the origin one gets (re-using the variables)
\begin{equation}\label{eq:ex1_diff_deg}
    \begin{split}
        r'&=\frac{a_1^2}{a_1-a_2}r\bx_2\\
        \bx_2'&=-\frac{a_1^2}{a_1-a_2}\bx_2^2,
    \end{split}
\end{equation}
which has linearization matrix $\begin{bmatrix}
    0 &0\\ 0 &0
\end{bmatrix}$, instead of the less degenerate form $\begin{bmatrix}
    0 &1\\ 0 &0
\end{bmatrix}$, which is mostly considered in the literature. Hence, for example, the techniques of \cite{menini2010stability} cannot be applied to achieve the stability of the origin via the introduction of higher-order terms. Of course, stabilization by modifying the weights, or equivalently by introducing new linear terms, is quite straightforward. If one were to consider weights
\begin{equation}
    w_{ij}=w_{ij}^*+\delta_{ij},
\end{equation}
in \eqref{eq:bu1r0}, one finds equilibria $(r^*,\bx_2^*)$ given by $r^*=0$ and
\begin{equation}
    \bx_2^*=\frac{\pm(a_2-a_1) \sqrt{4 a_1 \delta_{21}+4 a_2 \delta_{12}+(\delta_{12}+\delta_{21})^2}+2 a_1 a_2+a_1 \delta_{12}-a_1 \delta_{21}-a_2 \delta_{12}+a_2 \delta_{21}}{2 \left(a_1^2+a_1 \delta_{12}-a_2 \delta_{12}\right)}
\end{equation}
with corresponding Jacobian 
\begin{equation}
    J=\begin{bmatrix}
 \frac{1}{2} \left(\pm\sqrt{4 a_1 \delta_{21}+4 a_2 \delta_{12}+(\delta_{12}+\delta_{21})^2}-\delta_{12}-\delta_{21}\right) & 0 \\
 0 & \mp\sqrt{4 a_1 \delta_{21}+4 a_2 \delta_{12}+(\delta_{12}+\delta_{21})^2} \\
    \end{bmatrix}
\end{equation}
from which one can conclude, for example, that if $-\delta_{12}-\delta_{21}<0$ and $\delta_{}$ $4 a_1 \delta_{21}+4 a_2 \delta_{12}<0$ then there is one stable node and one saddle, and the flow on both $r$-eigenspaces is contracting. This then leads to a stable origin for \eqref{eq:ex_diff_eq0}, see figure \ref{fig:diff_st} for an example.
\begin{figure}[htbp]
    \centering
    \includegraphics{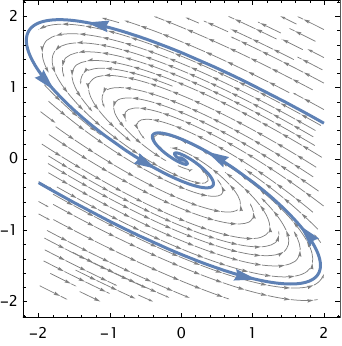}
    \caption{Phase-portrait of \eqref{eq:ex_diff_eq0} with $w_{ij}=w_{ij}^*+\delta_{ij}$ with $a_1=-2$, $a_2=1$, $\delta_{ij}=\frac{1}{5}$.}
    \label{fig:diff_st}
\end{figure}

It would be interesting to investigate, however, if higher-order interactions \cite{bick2023higher} can stabilize the origin. This is, in general, a very complicated problem and we limit ourselves to providing a proof of concept. Suppose one introduces higher-order terms as follows:
\begin{equation}\label{eq:ex_diff_eq0aux1}
    \begin{split}
        \dot x_1&=a_1x_1+w_{12}(x_2-x_1)-wx_2^3+w(x_1+x_2)^3\\
        \dot x_2&=a_2x_2+w_{21}(x_1-x_2)+wx_2^3,
    \end{split}
\end{equation}
where $w\in\R$. Upon a time-rescaling, one can assume that all parameters are small. To simplify the analysis, but keeping the problem still interesting, let $a_1=\alpha$ and $a_2=-\alpha$ with $\alpha>0$, hence the uncoupled dynamics are unstable. We emphasize that the origin is still nilpotent for \eqref{eq:ex_diff_eq0aux1} with $w_{ij}=w_{ij}^*$. Consider the polar blow-up
\begin{equation}
    (x,y)=(r\cos(\theta),r\sin(\theta))
\end{equation}
together with the rescaling of parameters $\alpha=r^3 A, w=r W$. In these new coordinates, and after division by $r^3$, we obtain the vector field
\begin{equation}\label{eq:app1}
    \begin{split}
        r'&= r f_1(\theta)\\
        \theta'&=f_2(\theta)
    \end{split}
\end{equation}
where, having substituted $w_{ij}=w_{ij}^*$, we have
\begin{equation}
    \begin{split}
        f_1(\theta) &=  \frac{1}{8} (4 A \cos (2 \theta )+W (6 \sin (2 \theta )+3 \sin (4 \theta )-\cos (4 \theta )+9))\\
        f_2(\theta) &= \frac{1}{8} (-2 (2 A+3 W) \sin (2 \theta )-4 A+W \sin (4 \theta )+3 W \cos (4 \theta )-3 W).
    \end{split}
\end{equation}
It follows that if we let\footnote{One can obtain the equilibria and their corresponding Jacobians numerically for any value of $A>0$ and $W<0$ leading to a similar qualitative behavior. It just happens that for this choice, the equilibria are independent of the parameters.} $W=-A$, then \eqref{eq:app1} has eight hyperbolic equilibria for $r=0$ given by 
\begin{equation}
    \{\theta_i\}\approx\{0.431808,  1.26918, 2.35619, 2.54775, 3.5734, 4.41077, 5.49779, 5.68935\},\;i=1,\ldots,8,
\end{equation}
where in particular $(r,\theta)=(0,\theta_k)$, $k=2,4,6,8$, are stable nodes and the rest are saddles with the $r$-direction stable. This leads to the diagram shown in figure \ref{fig:ex_hot_bu}. 

\begin{remark}
    Notice that thanks to the higher-order terms, the equilibria of the blown-up system are hyperbolic. That is, one can indeed desingularize a nilpotent equilibrium of a network dynamical system via the blow-up. 
\end{remark}

Due to hyperbolicity, a qualitatively similar result, as the one stated above, holds for $W\approx-A$. Moreover, since the result does not depend on the rescaling of the parameters, just on their signs, we conclude that the origin of \eqref{eq:ex_diff_eq0aux1} is stable for $a_1=\alpha$, $a_2=-\alpha$ with $\alpha>0$ and $w<0$, see a simulation in figure \ref{fig:ex_hot_bu}.
\begin{figure}
    \centering
    \begin{tikzpicture}
    \node at (3,7){
    \includegraphics[]{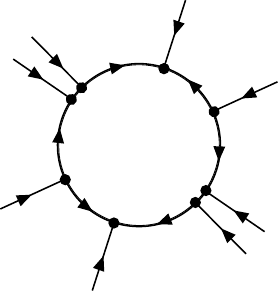}
    };
    \node at (3,4) {
        (a)
        };
    \node at (0,0) {
    \includegraphics[]{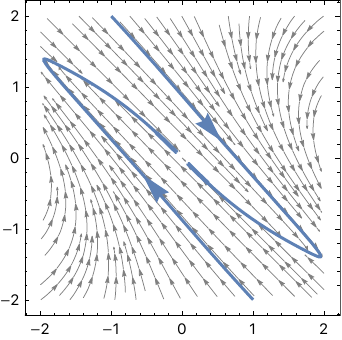}
    };
    
    \node at (0,-4){(b)};
    \node at (7,-4){(c)};
    \node at (7,0) {
    \includegraphics[]{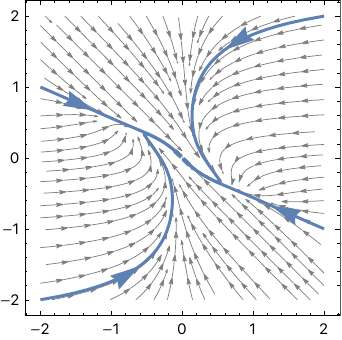}
    };
    \end{tikzpicture}
    \caption{(a) Global phase-portrait of \eqref{eq:app1}, where all equilibria are hyperbolic. Panels (b) and (c) show a simulation of \eqref{eq:ex_diff_eq0aux1} with $a_1=1$, $a_2=-1$, $w_{ij}=w_{ij}^*$, $w=-1/10$ and $w=-2$ respectively.}
    \label{fig:ex_hot_bu}
\end{figure}

The main conclusions of the previous analysis are summarized as follows.
\begin{proposition}
    Consider the 2-node network 
    \begin{equation}
    \begin{split}
        \dot x_1&=a_1x_1+w_{12}(x_2-x_1)-wx_2^3+w(x_1+x_2)^3\\
        \dot x_2&=a_2x_2+w_{21}(x_1-x_2)+wx_2^3,
    \end{split}
\end{equation}
with parameters satisfying \eqref{eq:diff_nil} so that the origin is nilpotent. If $w=0$ then the origin is unstable. On the other hand, if $a_1=-a_2=\alpha$ with $\alpha>0$ and $w<0$, then the origin is locally asymptotically stable.
\end{proposition}

The general case \eqref{eq:ex_diff_m} quickly becomes analytically intractable. For example, it is quite difficult to find closed expressions for the parameters such that the origin is nilpotent. Nevertheless, we conjecture that the dynamics shall be reminiscent of what is shown in figure \ref{fig:ex_diff_class}, and that as for the case $N=2$, higher-order interactions may stabilize nilpotent points. A special case that is simple to analyze is considered in the following proposition.

\begin{proposition}\label{prop:gradient} Consider
    \begin{equation}
    \dot x_i = a_ix_i+\sum_{j=1}^Nw_{ij}(x_j-x_i) + h_{i}(x),
\end{equation}
with $h_i(0)=0$, $h_{i}\in\cO(|x|^2)$, and with parameters $a_i\in\R$, $a_i\neq0$, and $w_{ij}\in\R$ such that the origin is nilpotent. Suppose that there exists a smooth function $V:\R^n\to\R$ such that $h(x)=(h_1(x),\ldots,h_n(x))=-\nabla V(x)$ and $x=0$ is a (degenerate, since $h_{i}\in\cO(|x|^2)$) local minimum of $V$. Then the origin is locally asymptotically stable.
\end{proposition}
\begin{proof}
    System \eqref{eq:ex_diff_m} can be rewritten as the gradient system $\dot x=-\nabla W$ where $W=(W_1,\ldots,W_n)$ with $W_i=-(a_i-w_{ii})x_i^2-\sum_{i=1,\,i\neq j}^nw_{ij}x_ix_j + V$. The statement is now straightforward from noticing that the nilpotent assumption guarantees that the origin is a (degenerate) local minimum of $W$.
\end{proof}

\begin{remark}
    Proposition \ref{prop:gradient} concerns network dynamical systems for which the leading part is nilpotent and the higher-order terms are of gradient form. The proposition roughly tells us that near the nilpotent point, the system behaves as a gradient system, hence the higher-order terms, if chosen appropriately, stabilize the nilpotent origin. Naturally, many network dynamical systems do not admit a gradient structure. In particular, one can check that \eqref{eq:ex_diff_eq0aux1} is not of such a gradient form.
\end{remark}

\subsection{An example of nilpotent internal dynamics}\label{sec:example_nilpotent_internal}

Consider the following system of weakly and diffusely coupled scalar nodes\footnote{This model can, alternatively, be studied with the tools developed in \cite{jardon2023persistent}. Here we present an analysis fitting the blow-up method.}
\begin{equation}\label{eq:ex_nil}
    \begin{split}
        \dot x_i &= a_ix_i^3(1-x_i)+\ve\sum_{j=1}^Nw_{ij}(x_j-x_i),
    \end{split}
\end{equation}
where $a_i\neq0$, $w_{ij}\in\R$, and for simplicity we assume that $x_i(0)\in(0,1)$, and that $w_{ii}=0$ for all $i=1,\ldots,N$ (there are no self-couplings). Since for $\ve=0$ the origin is nilpotent, we are going to use the blow-up technique to elucidate the stability of the origin for $\ve$ small.

A node directional blow-up given by
\begin{equation}
    (r,r\bx_1,\ldots,r\bx_{i-1},r\bx_{i+1},\ldots,r\bx_N,r^2\be)=(x_i,x_1,\ldots,x_{i-1},x_{i+1},\ldots,x_N,\ve)
\end{equation}
leads, after desingularization, to
\begin{equation}
    \begin{split}
        r' &= r\left( a_i+\be\sum_{j=1}^Nw_{ij}(\bx_j-1)-a_ir\right)\\
        \bx_k' &=a_k\bx_k^3+\be\sum_{j\neq i}^Nw_{kj}(\bx_j-\bx_k)+\be w_{ki}(1-\bx_k)-\bx_k\left( a_i+\be\sum_{j=1}^Nw_{ij}(\bx_j-1)-a_ir\right)-ra_k\bx_k^4\\
        \be' &=-2\be\left( a_i+\be\sum_{j=1}^Nw_{ij}(\bx_j-1)-a_ir\right),
    \end{split}
\end{equation}
$k=1,\ldots,N$ with $k\neq i$. In this model $x_i$ is replaced by $r$, hence one is interested in the stability of $r=0$, which we recall corresponds to the origin via the blow-up transformation. 

As is usual in the blow-up analysis, one is interested in the dynamics restricted to the invariant subspaces $\left\{ r=\be=0\right\}$, $\left\{\be=0\right\}$ and $\left\{ r=0\right\}$ (although in this case, it suffices to take the last two subspaces). For $\be=0$ we have
\begin{equation}
    \begin{split}
        r' &= a_ir(1-r)\\
        \bx_k' &=-a_i\bx_k+a_k\bx_k^3+r(a_i\bx_k-a_k\bx_k^4)\\
        \be' &=0,
    \end{split}
\end{equation}
where it is straightforward to see, from the leading order terms, that if $a_l<0$ for all $l=1,\ldots,N$ we have that, locally, $(r,\bx_k)=(0,0)$ is a \emph{hyperbolic saddle} while $(r,\bx_k)=(0,\pm\sqrt{a_i/a_k})$ are \emph{hyperbolic} sinks, with Jacobians
\begin{equation}
    J_{(0,0)}=\begin{bmatrix}
        a_i & 0 \\ 0 & -a_i
    \end{bmatrix}\qquad\textnormal{and}\qquad J_{\left(0,\pm\sqrt{a_i/a_k}\right)}=\begin{bmatrix}
        a_i & 0 \\ * & 2a_i
    \end{bmatrix}
\end{equation}
respectively.  It will be useful for our arguments below to notice that these stability properties depend only on the $a_i$ coefficient. 

Since the above equilibria are hyperbolic after the blow-up, one can conclude that for $\be>0$ small, the local stability properties of the equilibria are preserved. 

In the $\be$-direction (the rescalling chart) we obtain the desingularized system
\begin{equation}
    \begin{split}
        \dot \bx_i &= a_i\bx_i^3(1-\bx_i)+\sum_{j=1}^Nw_{ij}(\bx_j-\bx_i),
    \end{split}
\end{equation}
The leading order term is now the interaction. Therefore, the origin is not nilpotent anymore, but semi-hyperbolic. It follows that if all the weights $w_{ij}$ are nonnegative, and the digraph is strongly connected, then the subspace $\textnormal{span}\left\{1,\ldots,1 \right\}\in\R^N$ is locally attracting. The dynamics on such a space are thus given by the scalar equation
\begin{equation}
    \bx = a_i\bx^3(1-\bx),
\end{equation}
for which the origin is locally asymptotically stable provided that $a_i<0$.

From our previous analysis, we have proven the following:  
\begin{proposition}
    If $a_i<0$, $w_{ij}\leq0$ for all $i,j=1,\ldots,N$ and the underlying digraph of \eqref{eq:ex_nil} is strongly connected, then the (nilpotent) origin of \eqref{eq:ex_nil} is locally asymptotically stable for $\ve>0$ sufficiently small.
\end{proposition}

\subsection{A slowly adaptive network}\label{sec:example_adaptive}

Let us consider an adaptive network of $N$ homogeneous Kuramoto oscillators given by \cite{berner2018multi}
\begin{equation}
    \begin{split}
        \dot\phi_i &= \omega-\frac{1}{N}\sum_{j=1}^N\kappa_{ij}\sin(\phi_i-\phi_j+\alpha)\\
        \dot\kappa_{ij} &=-\ve(\sin(\phi_i-\phi_j+\beta)+\kappa_{ij}).
    \end{split}
\end{equation}

Without loss of generality, one may consider $\omega=0$ by changing to the co-rotating frame $\phi_i\mapsto\phi_i+\omega t$. Thus, from now on we consider
\begin{equation}\label{eq:ex_k_1}
    \begin{split}
        \dot\phi_i &= -\frac{1}{N}\sum_{j=1}^N\kappa_{ij}\sin(\phi_i-\phi_j+\alpha)\\
        \dot\kappa_{ij} &=-\ve(\sin(\phi_i-\phi_j+\beta)+\kappa_{ij}).
    \end{split}
\end{equation}

A one-cluster solution\footnote{Even though the one-cluster solutions for this system have already been described in \cite{berner2018multi}, we use it as a simple example to validate the application of the blow-up technique.} is defined as $\phi_i(t)=s(t)+a_i$ for some constant $a_i\in[0,2\pi)$, $i=1\ldots,N$. Depending on the phases $a_i$ the aforementioned one-cluster solutions receive different names, see \cite[Definition 2.3]{berner2018multi}. Upon substitution of the one-cluster solution in \eqref{eq:ex_k_1} one finds that $\dot s=$constant, meaning that one-cluster solutions are, in fact, of the form $\phi_i(t)=\Omega t+a_i$, with some $\Omega\in\R$. It is well-known that the one-cluster solutions with $\Omega=\Omega^*:=\frac{1}{N}\sum_{j}\sin(a_i-a_j+\alpha)\sin(a_i-a_j+\beta)$, together with $k_{ij}(t)=-\sin(a_i-a_j+\beta)$ form a set of relative equilibria of \eqref{eq:ex_k_1}. In other words, consider the co-rotating coordinates $\psi_i=\phi_i-(\Omega^* t+a_i)$, and the shifted weights $\sigma_{ij}=\kappa_{ij}+\sin(a_i-a_j+\beta)$. Then \eqref{eq:ex_k_1} is re-written as
\begin{equation}\label{eq:ex_k_2}
\begin{split}
    \dot\psi_i&=  -\Omega^*+\frac{1}{N}\sum_{j=1}^N\sin(a_i-a_j+\beta)\sin(\psi_i+a_i-\psi_j-a_j+\alpha)\\
    &\quad
    -\frac{1}{N}\sum_{j=1}^N\sigma_{ij}\sin(\psi_i+a_i-\psi_j-a_j+\alpha) \\
    \dot\sigma_{ij} &=-\ve(\sin(\psi_i+a_i-\psi_j-a_j+\beta)-\sin(a_i-a_j+\beta)+\sigma_{ij})
\end{split}
\end{equation}
It follows, indeed, that $(\psi_i^*,\sigma_{ij}^*)=(0,0)$ is an equilibrium of \eqref{eq:ex_k_2}. It turns out that this equilibrium is non-hyperbolic from a slow-fast perspective. Moreover, for certain parameters, one-cluster solutions can be nilpotent.
\begin{proposition}
    Consider the layer equation of \eqref{eq:ex_k_2}. Then the equilibrium point $(\psi_i^*,\sigma_{ij}^*)=(0,0)$ is non-hyperbolic. In particular, the antipodal solution, which corresponds to $a_i\in\left\{0,\pi\right\}$, is nilpotent for $\beta=0$. 
\end{proposition}
\begin{remark}
    These are not the only nilpotent solutions. In fact, for $\alpha=\beta=0$ all one-cluster solutions in \cite[Corollary 4.3]{berner2018multi} are nilpotent. In this example, however, we only concentrate on the antipodal case.
\end{remark}
\begin{proof}
    Since we are looking at the layer equation of \eqref{eq:ex_k_2}, we only need to consider the Jacobian $J=[J_{ij}]_{i,j=1,\ldots,N}$ where $J_{ij}=\frac{\partial\dot\psi_i}{\partial\psi_j}\Big|_{\psi_i=\psi_j=\sigma_{ij}=0}$. So, we have
    \begin{equation}
        \begin{split}
            J_{ii} &= \frac{1}{N}\sum_{j=1}^N\sin(a_i-a_j+\beta)\cos(a_i-a_j+\alpha)\\
            J_{ij} &=\frac{1}{N}\sin(a_i-a_j+\beta)\cos(a_i-a_j+\alpha).
        \end{split}
    \end{equation}
    Since $J_{ii}+\sum_{j\neq i}^N J_{ij}=0$, it follows that $\ker J$ is non-trivial (this is reminiscent, of course, of Laplacian matrices), showing that the equilibrium point $(\psi_i^*,\sigma_{ij}^*)=(0,0)$ is non-hyperbolic. Finally, if $\beta=0$ and $a_{i}\in\left\{ 0,\pi\right\}$ we have that $J=0$, i.e., the equilibrium point $(\psi_i^*,\sigma_{ij}^*)=(0,0)$ is nilpotent.
\end{proof}

The question one would now like to answer is: what is the stability of the anti-podal solutions for $\ve,\,\beta$ small? To answer this question we are going to use the blow-up technique. We remark that the blow-up technique is a local tool, hence, it is in principle applied to polynomial vector fields. Here we shall simply expand \eqref{eq:ex_k_2} for $(\psi_i,\sigma_{ij})$ close to the origin. Let $a_i\in\left\{ 0,\pi\right\}$. Up to a constant shift above, we can assume without loss of generality that $a_i=0$ for all $i=1,\ldots,N$. Then \eqref{eq:ex_k_2} reads as
\begin{equation}\label{eq:ex_k_3}
\begin{split}
    \dot\psi_i&=  -\Omega^*+\frac{1}{N}\sum_{j=1}^N\sin\beta\sin(\psi_i-\psi_j+\alpha)
    -\frac{1}{N}\sum_{j=1}^N\sigma_{ij}\sin(\psi_i-\psi_j+\alpha) \\
    \dot\sigma_{ij} &=-\ve(\sin(\psi_i-\psi_j+\beta)-\sin\beta+\sigma_{ij}).
\end{split}
\end{equation}
Expanding $\sin(\psi_i-\psi_j+\rho)$ for $\psi_i\approx \psi_j$, i.e. ($\sin(\psi_i-\psi_j+\rho)=\sin\rho+\cos\rho(\psi_i-\psi_j)+\cdots$), with $\rho=\alpha,\beta$, and accounting for the fact that in this setup we have $-\Omega^*+\sin\alpha\sin\beta=0$, we get
\begin{equation}\label{eq:ex_k_4}
\begin{split}
    \dot\psi_i&=\frac{1}{N}\sum_{j=1}^N\left[ \sin\beta\cos\alpha(\psi_i-\psi_j)-\sigma_{ij}(\sin\alpha+\cos\alpha(\psi_i-\psi_j) \right]+\cdots\\
    \dot\sigma_{ij}&=-\ve\left( \cos\beta(\psi_i-\psi_j)+\sigma_{ij}+\cdots\right),
\end{split}    
\end{equation}
where the $\cdots$ indicate higher-order terms. Naturally, the origin $(\psi_i,\sigma_{ij})=(0,0)$ is (still) nilpotent for $\ve=\beta=0$, but now it is straightforward to notice that $\alpha=\pm\frac{\pi}{2}$ also induces that the origin is nilpotent (this case is not further discussed in this example). For the rest of this example, we only consider the leading part of \eqref{eq:ex_k_4}.

Let us propose the $\beta$-directional blow-up (in the rescaling chart) given by\footnote{Notice that this is equivalent to first going to the rescaling chart and then blowing up in the $\pm\beta$-direction.}
\begin{equation}\label{eq:ex_k_bu}
    \psi_i=r u_i, \quad \sigma_{ij}=r s_{ij}, \quad \alpha=r A, \quad \beta=\pm r,\quad \ve=r.
\end{equation}
This leads to the desingularized vector field (where we take only the leading order terms of the Taylor expansion of the trigonometric functions)
\begin{equation}\label{eq:ex_k_5}
    \begin{split}
        \dot r &=0\\
        \dot u_i &= \frac{1}{N}\sum_{j=1}^N\pm(u_i-u_j)-s_{ij}(A+(u_i-u_j))\\
        \dot s_{ij} &= -((u_i-u_j)+s_{ij}).
    \end{split}
\end{equation}
Now, the following conclusions on the local stability of the line of equilibria $u_i=u_j,\,s_{ij}=0$, which we denote by $\gamma$, are straightforward\footnote{It suffices to consider the leading part $\dot u_i=\frac{\pm 1+A}{N}\sum_{i=1}^N(u_i-u_j).$}: a) if $\beta>0$, corresponding to the plus sign in \eqref{eq:ex_k_5}, then $\gamma$ cannot be stable; b) if $\beta<0$, corresponding to the negative sign in \eqref{eq:ex_k_5}, then $\gamma$ is locally stable for $A<1$ and unstable for $A>1$. From this analysis, and recalling the blow-up map \eqref{eq:ex_k_bu}, the following statement is proven (compare with \cite[figure 4.2 (b)]{berner2018multi}):
\begin{proposition}
    The antipodal solutions of \eqref{eq:ex_k_1} are locally asymptotically stable for $\beta<-\alpha$ and $\ve>0$ sufficiently small.
\end{proposition}

\section{Conclusions and Discussion}

We have studied the problem of structure preservation under the blow-up transformation of network dynamical systems, including adaptive ones. This transformation is a useful tool to rigorously analyze the dynamics of differential equations close to nilpotent equilibria. The essential philosophy of the method is to transform a singular problem into a regular one. 

The main conclusion from our analysis is that, indeed, the blow-up transformation preserves the network structure. This means that the local vector fields obtained after a directional blow-up can still be interpreted as a network dynamical system. The precise interpretation turns out to depend on the direction of blow-up and the model at hand. Moreover, via a series of examples, we have shown that (as in the classical low-dimensional setting) the blow-up helps to desingularize nilpotent singularities.   

Besides structure preservation, we have noticed that the blow-up induces parameters to become dynamic. Nevertheless, such dynamics seem to occur at higher-orders, which could yield another natural way to uncover higher-interactions in networks. In the particular case of edge-directional blow-up, including the case of adaptive networks, the main message is that singular networks with static topology are to be regularized as network dynamical systems with co-evolving edges; hence, this provides a route to potentially uncover adaptive network dynamics.

Several open problems stem from the analysis presented here, from which we highlight the following: 
\begin{enumerate}
    \item The problem of desingularization has been addressed here only in the examples. It would be interesting to relate the network topology and the nature of internal dynamics, interaction, and possible adaptation, to better describe the desingularization procedure via blow-up. In particular, one may want to study classes of networked systems where the corresponding vector field is quasi-homogeneous.
    \item For high-dimensional problems, it may become unrealistic to blow-up in all possible directions. This raises new directions for improving and better adapting the blow-up method for networks: a) one may want to blow-up in several directions simultaneously, b) one may want to devise techniques that distinguish between regular and singular directions, or c) one may want to develop symbolic computational tools that automatically provide the local vector fields in the relevant charts where the dynamics have been desingularized.
    \item Choosing the blow-up weights is, even for low dimensions, quite challenging. In the context of networks, it is not unfeasible to expect that the topology of the network induces, or forces, some patterns in the blow-up weights. For example, the analysis carried out in section \ref{sec:main_bu} suggests that, due to the network structure, all nodes are to be blown up with the same weight, provided that the linear part is not identically zero. In this regard, it is known that the quasi-homogeneous blow-up for planar systems can be related to a Newton polyhedron, see \cite{alvarez2011survey}. In higher dimensions, the quasihomogeneous blow-up should naturally be related to a polytope. It would be interesting to see if there is a relation between the blow-up polytope and the topology of the network to be desingularized. 
    \item A way to describe the dynamics of large networks is via mean field limits, which essentially provides a low-dimensional averaged description (e.g. via a PDE, but in certain cases even a low-dimensional system of ODEs) of the network. It would be interesting to investigate if a desingularization of a mean-field limit corresponds, in any way, to the desingularization of the corresponding large, but finite, network.
\end{enumerate}

\bibliographystyle{plainnat}
\bibliography{bib}

\end{document}